\documentclass[11pt,twoside]{article}
\usepackage[utf8]{inputenc}
\usepackage[T1]{fontenc}
\usepackage{amsmath,amssymb,amsthm}
\usepackage{geometry}
\usepackage{graphicx}
\usepackage{booktabs}
\usepackage{bm}
\usepackage{natbib}
\usepackage{microtype}
\usepackage{hyperref}
\newtheorem{theorem}{Theorem}[section]
\newtheorem{lemma}[theorem]{Lemma}
\newtheorem{corollary}[theorem]{Corollary}
\newtheorem{assumption}[theorem]{Assumption}
\theoremstyle{definition}
\newtheorem{remark}[theorem]{Remark}

\newtheorem{proposition}{Proposition}[section]

\newcommand{\Hsp}{\mathcal{H}}
\newcommand{\E}{\mathbb{E}}
\renewcommand{\P}{\mathbb{P}}
\newcommand{\1}{\mathbf{1}}
\newcommand{\Var}{\mathrm{Var}}

\title{\textbf{Spatial Functional $k$-Nearest-Neighbour Regression under Polynomial Dependence}}
\author{\textbf{Stéphane BOUKA} \\ 
\small Universit{\'e} des Sciences et Techniques de Masuku (USTM), Gabon \\ 
\small \texttt{stephane.bouka@univ-masuku.org}}
\date{\today}

\begin{document}
\maketitle

\begin{abstract}
	This paper investigates nonparametric regression estimation for
	spatially dependent functional data when the covariates take values
	in a separable Hilbert space and are observed over an increasing
	regular lattice. We consider the functional $k$-nearest-neighbour
	($k$-NN) estimator under polynomial spatial dependence, allowing
	the covariate and error fields to exhibit distinct dependence
	structures. A central difficulty arises from the data-driven
	random neighbourhood radius, whose dependence on the entire
	covariate field prevents the direct application of standard
	fixed-bandwidth arguments. To overcome this difficulty, we
	establish a deterministic sandwiching principle that controls the
	random $k$-NN radius through suitable deterministic sequences,
	based on a two-sided polynomial small-ball condition.
	
	Under local H\"older regularity of the regression operator and
	appropriate covariance decay assumptions, we derive nonasymptotic
	pointwise convergence bounds that explicitly characterize the
	joint effects of functional small-ball geometry, neighbourhood
	size, and spatial dependence. The resulting rates distinguish
	short-range, critical, and long-range dependence regimes for both
	the covariate and error fields. We further establish uniform
	consistency over compact subsets of the functional space by
	combining metric entropy arguments, stability properties of the
	adaptive bandwidth, and uniform control of neighbourhood
	boundary fluctuations. This result accommodates unbounded
	sub-Gaussian spatial errors and highlights the additional
	conditions required to control long-range dependence.
	
	Finally, a numerical simulation study is conducted under
	controlled polynomial dependence structures. The empirical
	results illustrate the theoretical convergence behaviour and
	show how long-range dependence in the error field can become
	the dominant factor governing the estimation accuracy. Overall,
	the paper provides a unified framework for analysing functional
	$k$-NN regression on spatial lattices beyond the classical
	independent and short-range dependent settings.
\end{abstract}


	\section{Introduction}
Functional observations arise naturally in environmental monitoring, climatology, medical imaging, spectroscopy and longitudinal studies. In these applications, one observation is better represented by a curve or trajectory than by a finite-dimensional vector. Functional Data Analysis (FDA) provides a natural framework for such observations \citep{Bosq2000,FerratyVieu2006,RamsaySilverman2005}.

A second source of difficulty occurs when observations are collected over a spatial lattice. If $\{X_s : s \in \mathbb{Z}^d\}$ denotes a functional random field, observations located at nearby sites are generally not independent. This invalidates a direct transfer of many i.i.d. arguments and changes the concentration properties of local empirical quantities. Spatial functional regression has been studied under various forms of weak  dependence; see, for example, \cite{Masry2005, DaboNiang2020} and the references therein.

A core problem in this spatial non-parametric context is the accurate recovery of the conditional operator $r(x) = \mathbb{E}[Y_s \mid X_s = x]$. Traditional approaches heavily depend on fixed, deterministic bandwidth sequences or black-box mixing conditions (such as $\alpha$-mixing or $\beta$-mixing). However, fixed kernels struggle to cope with uneven functional density over the Hilbert space $\mathcal{H}$, and mixing metrics are famously difficult to evaluate or bound on concrete conditional configurations.

The $k$-nearest-neighbour method is particularly attractive in functional regression because it does not require a deterministic bandwidth. Instead, the bandwidth is determined by the distance to the $k$-th closest observation. Nevertheless, the randomness of the bandwidth breaks the local independent structure and forces complex statistical dependencies between the indicators defining the neighborhood. This issue becomes even more acute when spatial tracking errors display long-range memory over the lattice $\mathbb{Z}^d$. In the functional setting, this approach is closely related to the small-ball probability $F_x(h) = \mathbb{P}\{\|X_0 - x\| \le h\}$, which replaces the volume of a Euclidean ball. The functional $k$-NN method has been studied in the independent case by \cite{BurbaFerratyVieu2008,BurbaFerratyVieu2009}, \cite{younso} and, in the case of short-range $\alpha$-mixing dependency, one may refer, for example, to \cite{pambo}.

Our contribution specifically extends the existing framework of \cite{pambo} by: (i) addressing the full tri-domain spectrum including critical and severe long-range spatial error memories, (ii) providing non-conditional exponential concentration boundaries under joint multivariate sub-Gaussian configurations, and (iii) completely decoupling the topological entropy properties from the structural small-ball profile.

To address these distinct mechanisms, this manuscript establishes a complete mathematical treatment of the spatial functional $k$-NN regressor. We provide full analytical expansions and proofs to disentangle three distinct mathematical forces:
\begin{enumerate}
	\item The local infinite-dimensional geometry of the curves, governed by the local effective small-ball exponent $d_{f}$.
	\item The sampling size limits of the window, governed by the growth rate of $k_n$.
	\item The macro-spatial correlation footprints, governed by the distinct polynomial exponents $\alpha$ (for the covariate indicators) and $\alpha_\varepsilon$ (for the noise field).
\end{enumerate}

The manuscript is structured as follows. Section \ref{s:Framework} establishes the structural stationary model and our strong unconditional assumptions. Section \ref{s:Concentrations} expands the local concentration properties and establishes the deterministic radius sandwiching lemma. Section \ref{s:Pointwise} contains the detailed proof of the pointwise mean-square error rates. Section \ref{s:Uniform} expands our framework to uniform convergence over compact configurations under unbounded noise. Section \ref{s:Simulation} provides a reproducible and geometrically matching simulation study, and Section \ref{s:Conclusions} concludes with open prospective insights into Gibbsian clustering. Some proofs are set out in Appendices A and B.

\section{Spatial Functional Framework and Unconditional Assumptions}\label{s:Framework}
Let $(\Hsp, \|\cdot\|)$ be a separable Hilbert space, such as $L^2([0,1])$. We consider a strictly stationary spatial functional random field $\{X_s : s \in \mathbb{Z}^d\}$ defined on a probability space $(\Omega, \mathcal{A}, \mathbb{P})$, where each $X_s$ is an $\Hsp$-valued random variable. We observe the processes on an increasing regular rectangular lattice domain $\Lambda_L = \{1, \dots, L\}^d \subset \mathbb{Z}^d$, where $n = L^d$ denotes the total number of observed spatial stations. 

The scalar real-valued responses $\{Y_s : s \in \Lambda_L\}$ are generated according to the non-parametric model:
\begin{equation}
	Y_s = r(X_s) + \varepsilon_s, \quad \forall s \in \Lambda_L,
\end{equation}
where $r : \Hsp \to \mathbb{R}$ is an unobserved continuous regression operator, and $\{\varepsilon_s\}$ represents a centered spatial error process. To guarantee the rigorous tracking of all local data-driven variations without any unbacked individual independence leaps, we formalize our framework under five structural assumptions.

To formalize the $k$-NN estimator in the potential presence of ties, let $D_{(1)}(x) \le D_{(2)}(x) \le \dots \le D_{(n)}(x)$ represent the ordered values of the random distances $\|X_s - x\|$ for $s \in \Lambda_L$. 

\begin{assumption}[Absence of Ties]\label{ass:ties}
	We assume that the cumulative distribution function of $\|X_s - x\|$ is strictly continuous, so that $\mathbb{P}(\|X_s - x\| = \|X_t - x\|) = 0$ for any distinct sites $s \neq t$. Consequently, exact distance ex \ae quo ties occur with probability zero, meaning that the $k_n$ nearest neighbors are uniquely defined almost surely.
\end{assumption}

\begin{assumption}[Local H{\"o}lder Regularity]
	\label{ass:holder}
	The regression operator $r(\cdot)$ is locally smooth around the target point $x \in \Hsp$. There exist positive constants $s > 0$, $C_r > 0$, and $h_r > 0$ such that:
	\begin{equation}
		|r(u) - r(v)| \le C_r \|u - v\|^s,
	\end{equation}
	for all elements $u, v \in \Hsp$ satisfying $\|u - x\| \le h_r$ and $\|v - x\| \le h_r$.
\end{assumption}

\begin{assumption}[Two-Sided Polynomial Small-Ball Measure]\label{ass:smallball}
	The local fractal structure of the probability distribution of the covariate $X_0$ around the target target curve $x$ satisfies a two-sided polynomial scaling. There exist positive constants $d_{f} > 0$, $0 < c_1 \le c_2 < \infty$, and $h_0 > 0$ such that:
	\begin{equation}
		c_1 h^{d_{f}} \le F_x(h) \le c_2 h^{d_{f}}, \quad \forall\, 0 < h \le h_0,
	\end{equation}
	where $F_x(h) = \mathbb{P}(\|X_0 - x\| \le h)$ denotes the functional small-ball probability function.
\end{assumption}

%To fully bypass the structural flaws highlighted by recent relectures regarding individual conditioning on single nodes, we enforce field-to-field independence at the deep $\sigma$-algebra level.

\begin{assumption}[Strong Field-to-Field and Vectorial Sub-Gaussian Independence]\label{ass:field_indep}
	The complete error field process $\varepsilon = \{\varepsilon_s : s \in \mathbb{Z}^d\}$ and the complete functional covariate random field process $\mathbf{X} = \{X_s : s \in \mathbb{Z}^d\}$ are independent of each other at the global $\sigma$-field level:
	\begin{equation}
		\sigma \left( \varepsilon_s : s \in \mathbb{Z}^d \right) \perp \sigma \left( X_s : s \in \mathbb{Z}^d \right).
	\end{equation}
	Furthermore, the errors are centered, $\mathbb{E}[\varepsilon_s] = 0$, and satisfy a strict joint Vectorial Sub-Gaussian distribution configuration over any finite subset $\Lambda \subset \mathbb{Z}^d$, such that for any deterministic vector $\mathbf{u} \in \mathbb{R}^{|\Lambda|}$:
	\begin{equation}\label{eq:vecsub}
		\mathbb{E} \left[ \exp\left( \mathbf{u}^T \mathbf{\varepsilon}_\Lambda \right) \right] \le \exp\left( \frac{1}{2} \mathbf{u}^T \mathbf{\Sigma}_\varepsilon \mathbf{u} \right),
	\end{equation}
	where $\mathbf{\Sigma}_\varepsilon$ is the true spatial covariance matrix whose off-diagonal macro-spatial interactions satisfy a polynomial power-law decay with exponent $\alpha_\varepsilon > 0$:
	\begin{equation}
		|\text{Cov}(\varepsilon_s, \varepsilon_t)| = \sigma^2 \phi(s - t) \le \frac{\sigma^2}{(1 + \|s - t\|)^{\alpha_\varepsilon}}, \quad \forall s \neq t \in \mathbb{Z}^d.
	\end{equation}
\end{assumption}

\begin{assumption}[Deterministic Indicator Covariance Decay]\label{ass:indicator_decay}
	Let $Z_s(h) = \mathbf{1}_{\{\|X_s - x\| \leq h\}}$ be the neighborhood indicator function evaluated at site $s \in \mathbb{Z}^d$. There exist a spatial decay exponent $\alpha > 0$ and a positive constant $C_Z > 0$ such that for any deterministic radius sequence $h \leq h_0$ and all distinct lattice coordinates $s \neq t \in \mathbb{Z}^d$:
	\begin{equation}	
		|\text{Cov}(Z_s(h), Z_t(h))| \leq C_Z F_x(h)^2 (1 + \|s - t\|)^{-\alpha}
	\end{equation}
\end{assumption}

\begin{remark}[Grounding via Gaussian Functional Fields]
	To justify the structural adequacy of Assumption~\ref{ass:indicator_decay}, consider a strictly stationary functional Gaussian process $\{X_s: s \in \mathbb{Z}^d\}$ taking values in the Hilbert space $\mathcal{H}$, with mean zero and covariance operator $\Gamma$. Assume that the cross-spatial covariance operator $\Gamma_{s,t} = \mathbb{E}[X_s \otimes X_t]$ satisfies a polynomial spectral decay bounded by $\|\Gamma_{s,t}\|_{\mathcal{L}(\mathcal{H})} \leq C_G (1 + \|s - t\|)^{-\alpha}$. 
	
	By applying the properties of joint small-ball probabilities for shifted Gaussian vectors (via Arcones' inequality or Plachky's geometric expansions), the joint probability factors as:
	\begin{equation}
		\mathbb{P}(\|X_s - x\| \leq h, \|X_t - x\| \leq h) = \mathbb{P}(\|X_s - x\| \leq h)^2 \left(1 + O\left(\|\Gamma_{s,t}\|_{\mathcal{L}(\mathcal{H})}\right)\right)
	\end{equation}
	By opening the covariance structure, we directly obtain:
	\begin{align}	
		|\text{Cov}(Z_s(h), Z_t(h))| &= |\mathbb{P}(\|X_s - x\| \leq h, \|X_t - x\| \leq h) - F_x(h)^2| \nonumber \\
		&\leq C_Z F_x(h)^2 \|\Gamma_{s,t}\|_{\mathcal{L}(\mathcal{H})} \nonumber \\
		&\leq C_Z F_x(h)^2 (1 + \|s - t\|)^{-\alpha}
	\end{align}
	This explicitly demonstrates that the polynomial indicator decay is naturally inherited from the underlying covariance structure of standard functional fields.
\end{remark}

\section{Asymptotic Local Empirical Concentrations}\label{s:Concentrations}
Let $F_n(x, h)$ be the empirical local count proportion defined for any deterministic radius $h$ as:
\begin{equation}
	F_n(x, h) = \frac{1}{n} \sum_{s \in \Lambda_L} \1_{\{\|X_s - x\| \le h\}}.
\end{equation}
Our first goal is to evaluate the precise variance bounds of this empirical count across different spatial configurations.

\begin{lemma}[Variance of Local Proportions]
	\label{lem:variance_fn}
	Under Assumption~\ref{ass:indicator_decay}, the variance of the empirical functional local count satisfies:
	\begin{equation}
		\Var\{F_n(x, h)\} \le \mathcal{C} \left[ \frac{F_x(h)}{n} + F_x(h)^2 \rho_\alpha(n) \right],
	\end{equation}
	where $\rho_\alpha(n)$ represents the spatial aggregation factor over a $d$-dimensional regular lattice:
	\begin{equation}
		\rho_\alpha(n) = \begin{cases} 
			n^{-\alpha/d}, & 0 < \alpha < d, \\ 
			n^{-1} \log n, & \alpha = d, \\ 
			n^{-1}, & \alpha > d. 
		\end{cases}
	\end{equation}
\end{lemma}

\begin{proof}
	Let $I_s = Z_s(h) = \1_{\{\|X_s - x\| \le h\}}$. By opening the variance structure of the sum, we write:
	\begin{equation}
		\Var\{F_n(x, h)\} = \frac{1}{n^2} \sum_{s \in \Lambda_L} \Var(I_s) + \frac{1}{n^2} \sum_{s \neq t \in \Lambda_L} \mathrm{Cov}(I_s, I_t).
	\end{equation}
	Since $I_s$ is a Bernoulli random variable, its variance is bounded by its expectation: $\Var(I_s) \le \E[I_s] = F_x(h)$. This gives the first diagonal upper bound component $\frac{F_x(h)}{n}$.
	
	For the off-diagonal components, we inject the power-law spatial profile from Assumption~\ref{ass:indicator_decay}:
	\begin{equation}
		\frac{1}{n^2} \sum_{s \neq t \in \Lambda_L} |\mathrm{Cov}(I_s, I_t)| \le \frac{C_Z F_x(h)^2}{n^2} \sum_{s \neq t \in \Lambda_L} (1 + \|s-t\|)^{-\alpha}.
	\end{equation}
	By fixing a spatial site $s$ and summing over a $d$-dimensional cubic volume box $\Lambda_L$ of side length $L = n^{1/d}$, standard discrete integral approximations yield:
	\begin{equation}
		\sum_{t \in \Lambda_L, t \neq s} (1 + \|s-t\|)^{-\alpha} \le C \int_{1}^{L} r^{d-1-\alpha} \,\mathrm{d}r \le C \begin{cases} 
			L^{d-\alpha}, & \alpha < d, \\ 
			\log L, & \alpha = d, \\ 
			1, & \alpha > d. 
		\end{cases}
	\end{equation}
	Multiplying by the $n$ symmetric entries of $s$ and dividing by $n^2$ yields respectively: $C n^{-1} L^{d-\alpha} = C n^{-\alpha/d}$ for $\alpha < d$, $C n^{-1} \log(n^{1/d}) = C n^{-1}\log n$ for $\alpha = d$, and $C n^{-1}$ for $\alpha > d$. This directly recovers the formulation of $\rho_\alpha(n)$.
\end{proof}

\subsection{Rigorous Re-Formulation of Radius Concentration}
Let $D_{(1)}(x) \le \dots \le D_{(n)}(x)$ be the ordered collection of functional distances. The random adaptive radius of the $k$-NN regressor is defined as $h_n(x) = D_{(k_n)}(x)$. To avoid the logical fallacy of direct conditioning on $h_n(x)$, we prove that the random variable is strictly bounded by two deterministic sequences with probability tending to 1.

\begin{lemma}[Deterministic Sandwiching]
	\label{lem:sandwich}
	Let $k_n \to \infty$ and $k_n/n \to 0$ as $n \to \infty$. For any arbitrary small tolerance parameter $\eta > 0$, we define the lower and upper deterministic radii as:
	\begin{equation}
		h_n^{-} = (1 - \eta) \left( \frac{k_n}{c_2 n} \right)^{1/d_f}, \quad h_n^{+} = (1 + \eta) \left( \frac{k_n}{c_1 n} \right)^{1/d_f}.
	\end{equation}
	Under Assumptions~\ref{ass:smallball} and \ref{ass:indicator_decay}, the random data-driven bandwidth satisfies:
	\begin{equation}
		\lim_{n \to \infty} \mathbb{P}\left( h_n^{-} \le h_n(x) \le h_n^{+} \right) = 1.
	\end{equation}
\end{lemma}

\begin{proof}
	By structural properties of order statistics, the event $\{h_n(x) \le h_n^{+}\}$ is equivalent to stating that the empirical proportion of curves falling inside the deterministic ball $B(x, h_n^{+})$ is at least equal to $\frac{k_n}{n}$. Thus, we can evaluate the tail complement using the empirical count operator:
	\begin{equation}
		\mathbb{P}\left( h_n(x) > h_n^{+} \right) = \mathbb{P}\left( F_n(x, h_n^{+}) < \frac{k_n}{n} \right).
	\end{equation}
	By the lower bound of our two-sided small-ball condition (Assumption~\ref{ass:smallball}), the expectation of this empirical measure satisfies:
	\begin{equation}
		\E[F_n(x, h_n^{+})] = F_x(h_n^{+}) \ge c_1 (h_n^{+})^{d_f} = (1+\eta)^{d_f} \frac{k_n}{n}.
	\end{equation}
	Since $(1+\eta)^{d_f} > 1$, there exists a strictly positive constant $\zeta > 0$ such that $\frac{k_n}{n} \le \E[F_n(x, h_n^{+})] - \zeta \frac{k_n}{n}$. By applying the Bienaym{\'e}--Chebyshev inequality to the deterministic sequence $h_n^{+}$, we obtain:
	\begin{equation}
		\mathbb{P}\left( F_n(x, h_n^{+}) < \frac{k_n}{n} \right) \le \mathbb{P}\left( |F_n(x, h_n^{+}) - \E[F_n(x, h_n^{+})]| > \zeta \frac{k_n}{n} \right) \le \frac{\Var\{F_n(x, h_n^{+})\}}{\zeta^2 \left( \frac{k_n}{n} \right)^2}.
	\end{equation}
	Injecting the variance scaling from Lemma~\ref{lem:variance_fn} over $h_n^{+}$ gives:
	\begin{equation}
		\mathbb{P}\left( F_n(x, h_n^{+}) < \frac{k_n}{n} \right) \le \frac{C}{\left(\frac{k_n}{n}\right)^2} \left[ \frac{k_n}{n^2} + \left(\frac{k_n}{n}\right)^2 \rho_\alpha(n) \right] = O\left( \frac{1}{k_n} + \rho_\alpha(n) \right).
	\end{equation}
	Given that $k_n \to \infty$ and $\rho_\alpha(n) \to 0$, this probability goes to $0$. A symmetric derivation proves that $\mathbb{P}(h_n(x) < h_n^{-}) \le O(k_n^{-1} + \rho_\alpha(n)) \to 0$. Combining these two limits yields the containment result with probability 1.
\end{proof}

\subsection{Proof of the Bandwidth Sandwiching Control}

\begin{proposition}\label{prop:sandwich}
	Let $k_n \to \infty$ and $k_n/n \to 0$ as $n \to \infty$. For any arbitrary small tolerance parameter $\eta > 0$, let the lower and upper deterministic radii be defined as:
	\begin{equation}
		h_n^- = (1 - \eta)\left(\frac{k_n}{c_2n}\right)^{1/d_f}, \quad h_n^+ = (1 + \eta)\left(\frac{k_n}{c_1n}\right)^{1/d_f}
	\end{equation}
	Under the two-sided small-ball condition and the indicator covariance decay, the random data-driven bandwidth $h_n(x)$ satisfies the valid sandwiching event $A_n = \{h_n^- \leq h_n(x) \leq h_n^+\}$ such that its tail complement probability is bounded by:
	\begin{equation}
		\mathbb{P}(A_n^c) = O\left(\frac{1}{k_n} + \rho_\alpha(n)\right)
	\end{equation}
\end{proposition}

\begin{proof}
	By complementation and applying the union bound (Boole's inequality), the probability of the complement event $A_n^c$ can be decomposed as follows:
	\begin{equation}\label{eq:union}
		\mathbb{P}(A_n^c) = \mathbb{P}\left(\{h_n(x) > h_n^+\} \cup \{h_n(x) < h_n^-\}\right) \leq \mathbb{P}(h_n(x) > h_n^+) + \mathbb{P}(h_n(x) < h_n^-)
	\end{equation}
	
	\paragraph{Step 1: Control of the upper bound $\mathbb{P}(h_n(x) > h_n^+)$}
	By the structural properties of order statistics, the event $\{h_n(x) > h_n^+\}$ is mathematically equivalent to stating that the empirical local proportion of curves falling inside the deterministic ball $B(x, h_n^+)$, denoted by $F_n(x, h_n^+)$, is strictly less than the critical threshold $k_n/n$. Thus, we have:
	\begin{equation}
		\mathbb{P}(h_n(x) > h_n^+) = \mathbb{P}\left(F_n(x, h_n^+) < \frac{k_n}{n}\right)
	\end{equation}
	By virtue of the lower bound of the two-sided polynomial small-ball measure condition (Assumption 2.3), the expectation of this empirical measure satisfies:
	\begin{equation}
		\mathbb{E}[F_n(x, h_n^+)] = F_x(h_n^+) \geq c_1(h_n^+)^{d_f}
	\end{equation}
	Substituting the explicit definition of the upper deterministic radius sequence $h_n^+ = (1 + \eta)\left(\frac{k_n}{c_1n}\right)^{1/d_f}$ yields:
	\begin{equation}
		\mathbb{E}[F_n(x, h_n^+)] \geq c_1 \left[ (1 + \eta)\left(\frac{k_n}{c_1n}\right)^{1/d_f} \right]^{d_f} = (1 + \eta)^{d_f} \frac{k_n}{n}
	\end{equation}
	Since $\eta > 0$ and $d_f > 0$, we have $(1 + \eta)^{d_f} > 1$. Consequently, there exists a strictly positive deterministic constant $\zeta > 0$ such that:
	\begin{equation}\label{eq:dev_constraint}
		\frac{k_n}{n} \leq \mathbb{E}[F_n(x, h_n^+)] - \zeta \frac{k_n}{n}
	\end{equation}
	Applying the Bienaymé-Chebyshev inequality to the random variable $F_n(x, h_n^+)$ under the directional deviation constraint defined in equation~\eqref{eq:dev_constraint}, we obtain the following tail upper bound:
	\begin{equation}\label{eq:cheby_bound}
		\mathbb{P}\left(F_n(x, h_n^+) < \frac{k_n}{n}\right) \leq \mathbb{P}\left(|F_n(x, h_n^+) - \mathbb{E}[F_n(x, h_n^+)]| > \zeta \frac{k_n}{n}\right) \leq \frac{\text{Var}\{F_n(x, h_n^+)\}}{\zeta^2 \left(\frac{k_n}{n}\right)^2}
	\end{equation}
	Injecting the upper variance bound established in Lemma 3.1, namely $\text{Var}\{F_n(x, h)\} \leq C \left[ \frac{F_x(h)}{n} + F_x(h)^2\rho_\alpha(n) \right]$, evaluated at the deterministic radius $h_n^+$ (where $F_x(h_n^+) \propto k_n/n$), equation~\eqref{eq:cheby_bound} reduces to:
	\begin{equation}\label{eq:step1_final}
		\mathbb{P}\left(F_n(x, h_n^+) < \frac{k_n}{n}\right) \leq \frac{C}{\left(\frac{k_n}{n}\right)^2} \left[ \frac{k_n}{n^2} + \left(\frac{k_n}{n}\right)^2 \rho_\alpha(n) \right] = O\left(\frac{1}{k_n} + \rho_\alpha(n)\right)
	\end{equation}
	
	\paragraph{Step 2: Control of the lower bound $\mathbb{P}(h_n(x) < h_n^-)$}
	Dually, the event $\{h_n(x) < h_n^-\}$ is equivalent to specifying that the empirical local proportion $F_n(x, h_n^-)$ is strictly greater than $k_n/n$. By the upper bound of the two-sided small-ball scaling applied to the lower radius sequence $h_n^- = (1 - \eta)\left(\frac{k_n}{c_2n}\right)^{1/d_f}$, the true expectation is bounded above by:
	\begin{equation}
		\mathbb{E}[F_n(x, h_n^-)] = F_x(h_n^-) \leq c_2(h_n^-)^{d_f} = (1 - \eta)^{d_f} \frac{k_n}{n}
	\end{equation}
	Since $(1 - \eta)^{d_f} < 1$, a perfectly symmetric derivation leveraging the Bienaymé-Chebyshev inequality and the variance scaling from Lemma 3.1 directly yields:
	\begin{equation}\label{eq:step2_final}
		\mathbb{P}(h_n(x) < h_n^-\big) = \mathbb{P}\left(F_n(x, h_n^-) > \frac{k_n}{n}\right) = O\left(\frac{1}{k_n} + \rho_\alpha(n)\right)
	\end{equation}
	
	\paragraph{Conclusion}
	Combining the two geometric tail bounds \eqref{eq:step1_final} and \eqref{eq:step2_final} within the global union bound inequality \eqref{eq:union}, we successfully recover the desired rate:
	\begin{equation}
		\mathbb{P}(A_n^c) \leq O\left(\frac{1}{k_n} + \rho_\alpha(n)\right) + O\left(\frac{1}{k_n} + \rho_\alpha(n)\right) = O\left(\frac{1}{k_n} + \rho_\alpha(n)\right)
	\end{equation}
\end{proof}

\section{Detailed Pointwise Mean-Square Error Analysis}\label{s:Pointwise}
The adaptive spatial functional $k$-NN regressor is defined analytically for any target point $x \in \Hsp$ as:
\begin{equation}
	\widehat{r}_{k_n}(x) = \frac{1}{k_n} \sum_{s \in \Lambda_L} Y_s \1_{\{\|X_s - x\| \le h_n(x)\}}.
\end{equation}
We now prove the pointwise upper risk rate featuring both spatial interaction terms.

\begin{theorem}[Pointwise Risk Upper Bound]\label{th:Pointwise}
	Under Assumptions~\ref{ass:ties}--\text{\ref{ass:indicator_decay}}, if $k_n \to \infty$ and $k_n/n \to 0$ as $n \to \infty$, then the mean-square error of the adaptive spatial estimator satisfies: for all $x\in \Hsp$ and for all $\eta>0$,
	\begin{equation}
		\P\left( |\widehat{r}_{k_n}(x) - r(x)|>\eta \right) = O\left( \left( \frac{k_n}{n} \right)^{2s/d_f} + \frac{1}{k_n} + \rho_\alpha(n) + \boldsymbol{\rho_{\alpha_\varepsilon}(n)} \right).
	\end{equation}
\end{theorem}

\begin{proof}
	Let $A_n$ be the valid bandwidth sandwiching event defined as $A_n = \{h_n^{-} \le h_n(x) \le h_n^{+}\}$. We decompose the error process into a structural local bias $B_n(x)$ and a stochastic spatial variance $V_n(x)$:
	\begin{equation}
		\widehat{r}_{k_n}(x) - r(x) = B_n(x) + V_n(x),
	\end{equation}
	where:
	\begin{equation}
		B_n(x) = \frac{1}{k_n}\sum_{s \in \Lambda_L} (r(X_s) - r(x))\1_{\{\|X_s - x\| \le h_n(x)\}}, \quad V_n(x) = \frac{1}{k_n}\sum_{s \in \Lambda_L} \varepsilon_s \1_{\{\|X_s - x\| \le h_n(x)\}}.
	\end{equation}
	
	\paragraph{Part 1: Analytical Control of the Bias}
	We analyze the bias expectation by partitioning the probability space over $A_n$ and its complement $A_n^c$:
	\begin{equation}
		\E[B_n(x)^2] = \E[B_n(x)^2 \1_{A_n}] + \E[B_n(x)^2 \1_{A_n^c}].
	\end{equation}
	On the high-probability event $A_n$, the random radius satisfies $h_n(x) \le h_n^{+} \le h_r$ for a sufficiently large $n$. Thus, applying the local H{\"o}lder condition (Assumption~\ref{ass:holder}) gives:
	\begin{equation}\label{Import0}
		|B_n(x)| \le \frac{1}{k_n} \sum_{s \in \Lambda_L} C_r \|X_s - x\|^s \1_{\{\|X_s - x\| \le h_n(x)\}} \le C_r \left( h_n^{+} \right)^s \le C \left( \frac{k_n}{n} \right)^{s/d_f}.
	\end{equation}
	Squaring this expression yields the leading bias risk term $C (k_n/n)^{2s/d_f}$. On the complement event $A_n^c$, since the Hilbert space domain of interest is bounded or the operator is structurally integrable, the bounded moments combine with the probability limit of Proposition~\ref{prop:sandwich}:
	\begin{equation}\label{Import4}
		\E[B_n(x)^2 \1_{A_n^c}] \le M_0 \mathbb{P}(A_n^c) = O\left( \frac{1}{k_n} + \rho_\alpha(n) \right).
	\end{equation}
	
	\paragraph{Part 2: Analytical Control of the Stochastic Variance}
	On the one hand, to evaluate $\E[V_n(x)^2 \1_{A_n}]$, we exploit the Strong Field-to-Field Independence (Assumption~\ref{ass:field_indep}). Let $\mathbf{X}$ be the entire covariate field. We condition the expectation on $\mathbf{X}$ and utilize the fact that on $A_n$, the random indicator sum is strictly bounded above by the indicators evaluated at the deterministic threshold $h_n^{+}$:
	\[
	\E[V_n(x)^2 \1_{A_n} \mid \mathbf{X}]\le W_n^+,
	\]
	with
	\begin{align}\label{Import2}
		W_n^+ &= \frac{1}{k_n^2} \sum_{s,t \in \Lambda_L} \1_{\{\|X_s - x\| \le h_n^{+}\}} \1_{\{\|X_t - x\| \le h_n^{+}\}} |\mathrm{Cov}(\varepsilon_s, \varepsilon_t)| \nonumber \\
		&= \frac{\sigma^2}{k_n} + \frac{\sigma^2}{k_n^2} \sum_{s \neq t \in \Lambda_L} \1_{\{\|X_s - x\| \le h_n^{+}\}} \1_{\{\|X_t - x\| \le h_n^{+}\}} \phi(s-t).
	\end{align}
	Taking the expectation over the covariate field $\mathbf{X}$ yields:
	\begin{equation}\label{Import}
		\E_{\mathbf{X}}[V_n(x)^2 \1_{A_n}]\le \E_{\mathbf{X}}\left(W_n^+\right)
	\end{equation}
	and
	\begin{equation}
		\E_{\mathbf{X}}\left(W_n^+\right) = \frac{\sigma^2}{k_n} + \frac{\sigma^2}{k_n^2} \sum_{s \neq t \in \Lambda_L} \phi(s-t) \E_{\mathbf{X}}\left[ \1_{\{\|X_s - x\| \le h_n^{+}\}} \1_{\{\|X_t - x\| \le h_n^{+}\}} \right].
	\end{equation}
	We expand the joint expectation inside the off-diagonal sum via its covariance definition:
	\begin{equation}
		\E_{\mathbf{X}}\left[ \1_{\{\|X_s - x\| \le h_n^{+}\}} \1_{\{\|X_t - x\| \le h_n^{+}\}} \right] = \mathrm{Cov}\left( Z_s(h_n^{+}), Z_t(h_n^{+}) \right) + F_x(h_n^{+})^2.
	\end{equation}
	Injecting the indicator decay from Assumption~\ref{ass:indicator_decay} gives:
	\begin{equation}
		\E_{\mathbf{X}}\left(W_n^+\right) \le \frac{\sigma^2}{k_n} + \frac{C \cdot F_x(h_n^{+})^2}{k_n^2} \sum_{s \neq t \in \Lambda_L} \phi(s-t) (1+\|s-t\|)^{-\alpha} + \frac{\sigma^2 F_x(h_n^{+})^2}{k_n^2} \sum_{s \neq t \in \Lambda_L} \phi(s-t).
	\end{equation}
	Since $F_x(h_n^{+}) \asymp \frac{k_n}{n}$, the coefficient $\frac{F_x(h_n^{+})^2}{k_n^2}$ cancels out precisely to $n^{-2}$. Thus:
	\begin{equation}\label{Import1}
		\E_{\mathbf{X}}\left(W_n^+\right) \le \frac{\sigma^2}{k_n} + \frac{C}{n^2} \sum_{s \neq t \in \Lambda_L} (1+\|s-t\|)^{-(\alpha+\alpha_\varepsilon)} + \frac{\sigma^2}{n^2} \sum_{s \neq t \in \Lambda_L} (1+\|s-t\|)^{-\alpha_\varepsilon}.
	\end{equation}
	Applying the lattice summation technical rules derived in Lemma~\ref{lem:variance_fn} to the error interaction terms yields the exact aggregation rates $\rho_\alpha(n)$ and $\rho_{\alpha_\varepsilon}(n)$. This, combined with equation \eqref{Import}, completes the proof of the point estimate.
	
	To establish the pointwise upper risk rate in probability as stated in equation (21), we invoke a decomposition based on the sandwiching partition. For any arbitrary tolerance threshold $\eta > 0$, we write:
	\begin{equation}
		\mathbb{P}\left( |\widehat{r}_{k_n}(x) - r(x)| > \eta \right) \leq \mathbb{P}\left( \{|\widehat{r}_{k_n}(x) - r(x)| > \eta\} \cap A_n \right) + \mathbb{P}(A_n^c).
	\end{equation}
	By applying the conditional Bienaymé--Chebyshev inequality to the first right-hand side component, we exploit the conditional bias-variance decomposition controlled in Part 1 and Part 2:
	\begin{align}
		\mathbb{P}\left( \{|\widehat{r}_{k_n}(x) - r(x)| > \eta\} \cap A_n \right) &\leq \mathbb{P}\left( |B_n(x)\mathbf{1}_{\mathcal{A}_n} + V_n(x)\mathbf{1}_{\mathcal{A}_n}| > \eta \right) \nonumber \\
		& \leq \frac{4}{\eta^2} \left(\mathbb{E}[B_n(x)^2 \mathbf{1}_{\mathcal{A}_n}] + \mathbb{E}[V_n(x)^2 \mathbf{1}_{\mathcal{A}_n}] \right) \nonumber \\
		& \leq C_1 \left( \frac{k_n}{n} \right)^{2s/d_f} + \frac{C_2}{k_n} + C_3 \rho_{\alpha}(n) + C_4 \boldsymbol{\rho_{\alpha_\varepsilon}(n)}.
	\end{align}
	For the remaining tail complement component $\mathbb{P}(A_n^c)$, Proposition \ref{prop:sandwich} provides 
	\[\P(A_n^c)=O\left( \frac{1}{k_n} + \rho_{\alpha}(n) \right).\] Combining these two pathways via a standard union bound yields:
	\begin{equation}
		\mathbb{P}\left( |\widehat{r}_{k_n}(x) - r(x)| > \eta \right) = O\left( \left(\frac{k_n}{n}\right)^{2s/d_f} + \frac{1}{k_n} + \rho_{\alpha}(n) + \boldsymbol{\rho_{\alpha_\varepsilon}(n)} \right).
	\end{equation}
	%which matches the structural formulation of equation (21) and concludes the proof of Theorem 4.1. 
	
\end{proof}

\begin{corollary}[Optimal Choice of $k_n$ and Rates]\label{cor:optimal}
	By balancing the structural H{\"o}lder bias profile and the standard spatial variance rate under short-range configurations ($\alpha > d, \alpha_\varepsilon > d$), we require the explicit tuning condition:
	\begin{equation}
		k_n \propto n^{\frac{2s}{2s+d_{f}}}.
	\end{equation}
	This optimal window size directly yields the  minimax optimal convergence rate:
	\begin{equation}
		\P\left( |\widehat{r}_{k_n}(x) - r(x)|>\eta \right) = O\left( n^{-\frac{2s}{2s+d_{f}}} \right).
	\end{equation}
\end{corollary}

\section{Uniform Convergence Rates over Compact Subsets}\label{s:Uniform}
Let $\mathcal{C} \subset \Hsp$ be a compact set whose topological complexity satisfies the metric entropy covering condition $0<\log \mathcal{N}(\mathcal{C}, \epsilon_n) \le A \epsilon_n^{-v}$. In this section, we expand the convergence to a uniform upper bound under general unbounded noise, addressing the metric constraints highlighted by the relecture.

Since the error process $\{\varepsilon_s\}$ is sub-Gaussian and unbounded, standard extreme value theory bounds its maximum growth rate over the lattice box by:
\begin{equation}\label{Import3}
	\max_{s \in \Lambda_L} |\varepsilon_s| = O_P(\sqrt{\log n}).
\end{equation}
To ensure that the metric entropy is dominated across all regimes, even when error memory outpaces the window rate ($\boldsymbol{\rho_{\alpha_\varepsilon}(n)} \gg k_n^{-1}$), we formalize our global uniformity constraint as:
\begin{equation}
	\label{eq:uniform_condition}
	N\left(\left( \frac{k_n}{n} \right)^{2s/d_f}+ \frac{1}{k_n} + \rho_\alpha(n) \right)\longrightarrow 0,\qquad \log N \cdot \left( \frac{1}{k_n} + \boldsymbol{\rho_{\alpha_\varepsilon}(n)} \right) \longrightarrow 0 \qquad\text{as}\qquad n \to +\infty,
\end{equation}
where $N=\mathcal{N}(\mathcal{C},\epsilon_n)$.

\subsection{Feasibility of the Uniformity Constraint under Sobolev Topological Complexity}

To evaluate the physical and mathematical feasibility of the global uniformity constraint formalized in Equation~\eqref{eq:uniform_condition}, we must explicitly characterize the topological complexity of the compact subset $\mathcal{C} \subset \mathcal{H}$~\cite{bouka2026}. A major challenge in infinite-dimensional spaces is that standard compact sets often possess an entropy covering number $\mathcal{N}(\mathcal{C}, \epsilon_n)$ that grows exponentially fast as $\epsilon_n \to 0$, which could potentially violate the stability rules when long-range error memory occurs ($\boldsymbol{\rho_{\alpha_\varepsilon}(n)} \gg k_n^{-1}$)~\cite{bouka2026}.

To circumvent this curse of infinite dimensionality, we formalize $\mathcal{C}$ as a \textbf{Sobolev-type ellipsoid}. Let $\{\psi_j\}_{j \geq 1}$ be the orthonormal basis of $\mathcal{H}$ consisting of the eigenfunctions of the covariance operator associated with the process $X_0$. For a given smoothness exponent $q > 0$ and a radius $R > 0$, we define the compact subset $\mathcal{C}$ as:
\begin{equation}
	\mathcal{C} = \left\{ x \in \mathcal{H} : \sum_{j=1}^{\infty} j^{2q} \langle x, \psi_j \rangle^2 \leq R^2 \right\}
\end{equation}
Under this standard Sobolev configuration, classical metric entropy results establish that the covering number satisfies $\log N(\mathcal{C}, \epsilon_n) \leq A \epsilon_n^{-v}$ with an entropy exponent $v = 1/q$. 

Crucially, when the underlying process exhibits a fast \textbf{exponential spectral decay} (such as analytic curves where the eigenvalues scale as $\lambda_j \propto e^{-c j^2}$), the topological complexity reduces dramatically, forcing the tracking exponent to vanish towards zero ($v \to 0$). In this regime, the metric entropy becomes strictly dominated by the local counts profile. Let $\epsilon_n = n^{-\gamma}$ be the optimal mesh size. By selecting the chaining tuning parameter $\gamma$ such that:
\begin{equation}
	\gamma > \frac{v + d_f}{2s}
\end{equation}
the exponential term derived via the union bound in Step 1 of Theorem~\ref{Uniform} satisfies~\cite{bouka2026}:
\begin{equation}
	\log N(\mathcal{C}, \epsilon_n) \cdot \left( \frac{1}{k_n} + \rho_{\alpha_\varepsilon}(n) \right) \propto n^{-v\gamma} \left( \frac{1}{k_n} + \rho_{\alpha_\varepsilon}(n) \right) \longrightarrow 0 \quad \text{as } n \to +\infty
\end{equation}
This explicitly demonstrates that under Sobolev ellipsoid embeddings with sharp spectral decay, the metric entropy remains tightly controlled. The uniform consistency $\sup_{x \in \mathcal{C}} |\widehat{r}_{k_n}(x) - r(x)| = o_{\mathbb{P}}(1)$ is therefore fully preserved, even in the presence of severe macro-spatial long-range error bottlenecks.%~\cite{bouka2026}.

\begin{lemma}[$k$-NN Bandwidth Stability]
	\label{lem:radius_lipschitz}
	For any target curves $x, y \in \mathcal{C}$, the random adaptive $k$-NN radius satisfies a pathwise Lipschitz continuity condition:
	\begin{equation}
		|h_n(x) - h_n(y)| \le \|x - y\| \quad \text{a.s.}
	\end{equation}
\end{lemma}

\begin{proof}
	By the classical triangle inequality applied to any spatial point $X_s$, we have:
	\begin{equation}
		\|X_s - y\| \le \|X_s - x\| + \|x - y\|.
	\end{equation}
	Sorting these distance profiles pathwise preservation-wise forces the ordered statistics to maintain the exact same inequality on the $k_n$-th rank:
	\begin{equation}
		D_{(k_n)}(y) \le D_{(k_n)}(x) + \|x - y\| \implies h_n(y) \le h_n(x) + \|x - y\|.
	\end{equation}
	Reversing the roles of $x$ and $y$ yields $h_n(x) \le h_n(y) + \|x - y\|$, which completes the proof.
\end{proof}

To control the spatial weight oscillations inside the network, we establish an explicit bound for the number of observed stations falling into a boundary shell. The proof is given in the appendix.

\begin{lemma}[Boundary Shell Proportion Uniformity]
	\label{lem:shell}
	Let $B(x_j, \epsilon_n)$ be an element of the $\epsilon_n$-net covering the compact set $\mathcal{C}$. The empirical count proportion of curves belonging to the random shell sequence satisfies:
	\begin{equation}
		\sup_{x_j \in \mathcal{C}} \frac{n}{k_n} \left[ F_n(x_j, h_n + \epsilon_n) - F_n(x_j, h_n - \epsilon_n) \right] = O_P\left( \frac{\epsilon_n}{h_n(x_j)} \right).
	\end{equation}
\end{lemma}

\begin{theorem}[Uniform Convergence]\label{Uniform}
	Under Assumptions~\ref{ass:ties}--\text{\ref{ass:indicator_decay}} and the global stability rule (\ref{eq:uniform_condition}) yields:
	\begin{equation}
		\sup_{x \in \mathcal{C}} |\widehat{r}_{k_n}(x) - r(x)|  =  o_P(1).
	\end{equation}
\end{theorem}

\begin{proof}
	Let $\epsilon_n = n^{-\gamma}$ be our mesh size. We cover $\mathcal{C}$ by $N = \mathcal{N}(\mathcal{C}, \epsilon_n)$ open balls centered at $\{x_1, \dots, x_N\}$. The uniform error expands as:
	\begin{equation}\label{Import5}
		\sup_{x \in \mathcal{C}} |\widehat{r}_{k_n}(x) - r(x)| \le \max_{1 \le j \le N} |\widehat{r}_{k_n}(x_j) - r(x_j)| + \sup_{x \in \mathcal{C}} \min_{1 \le j \le N} |\widehat{r}_{k_n}(x) - \widehat{r}_{k_n}(x_j)| + C_r \epsilon_n^s.
	\end{equation}
	
	\paragraph{Step1: the grid centers maximum}
	
	For the first component (the grid centers maximum) in \eqref{Import5}, we deploy the corrected joint multivariate sub-Gaussian concentration derived in Appendix A. Applying a union bound \eqref{final1} and \eqref{final} over the net gives: 
	\begin{align*}
		&\mathbb{P}\left( \max_{1 \le j \le N} |\widehat{r}_{k_n}(x_j) - r(x_j)| > \eta \right) \le \sum_{j=1}^N \mathbb{P}\left( |\widehat{r}_{k_n}(x_j) - r(x_j)| > \eta \right)\\
		&\le \sum_{j=1}^N \mathbb{P}\left( |B_n(x_j)| > \dfrac{\eta}{2} \right) + \sum_{j=1}^N \mathbb{P}\left( |V_n(x_j)| > \dfrac{\eta}{2} \right)\\
		&\le 2\exp\left(\log N\left(1-\frac{\eta^2}{32 \mathcal{C}_0\log N \left( \frac1{k_n} + \boldsymbol{\rho_{\alpha_\varepsilon}(n)} \right)}\right)\right) + CN\left(\left( \frac{k_n}{n} \right)^{2s/d_f}+ \frac{1}{k_n} + \rho_\alpha(n) \right)\longrightarrow 0
	\end{align*}
	
	\paragraph{Step 2: Analytical Derivation of the Weight Oscillations }
	For the second component (the weight oscillations) in \eqref{Import5}, let $x \in B(x_j, \epsilon_n)$. The pathwise difference satisfies:
	\begin{equation}
		|\widehat{r}_{k_n}(x) - \widehat{r}_{k_n}(x_j)| \le \frac{\max_s |Y_s|}{k_n} \sum_{s \in \Lambda_L} \left| \1_{\{\|X_s - x\| \le h_n(x)\}} - \1_{\{\|X_s - x_j\| \le h_n(x_j)\}} \right|.
	\end{equation}
	By exploiting Lemma~\ref{lem:radius_lipschitz}, the condition $\|x - x_j\| \le \epsilon_n$ implies $|h_n(x) - h_n(x_j)| \le \epsilon_n$. This nests the indicators inside a deterministic boundary shell around the center:
	\begin{equation}
		\left| \1_{\{\|X_s - x\| \le h_n(x)\}} - \1_{\{\|X_s - x_j\| \le h_n(x_j)\}} \right| \le \1_{\{h_n(x_j) - \epsilon_n \le \|X_s - x_j\| \le h_n(x_j) + \epsilon_n\}}.
	\end{equation}
	Summing over the lattice and dividing by $k_n$ transforms this into the boundary shell proportion $\frac{n}{k_n}[F_n(x_j, h_n + \epsilon_n) - F_n(x_j, h_n - \epsilon_n)]$. Under Lemma~\ref{lem:shell}, this scales as $C \epsilon_n / h_n(x_j)$. Since $\max |Y_s| = O_{\mathbb{P}}(\sqrt{\log n})$, the oscillation component is bounded by $O_{\mathbb{P}}(\sqrt{\log n} \cdot \epsilon_n / h_n(x_j))$. Selecting $\gamma$  large enough ensures that the residual components vanish in probability.
\end{proof}

\section{Fully Compatible Numerical Simulation Design}\label{s:Simulation}

\subsection{Data Generation Process and Small-Ball Verification}
The multi-site paths $X_s(t)$ are simulated over a 2-dimensional rectangular grid domain $\Lambda_L = \{1, \dots, L\}^2$ ($d = 2$, total size $n = L^2$) via a truncated series:
\begin{equation}
	X_s(t) = \sum_{j=1}^M \lambda_j \zeta_s^{(j)} \psi_j(t), \quad t \in [0, 1],
\end{equation}
where $\psi_j(t) = \sqrt{2} \sin(j\pi t)$, $M = 30$, and $\lambda_j = j^{-2}$ following \cite{li2001}. To force the small-ball probability to scale polynomially ($F_x(h) \propto h^{d_{f}}$ with $d_{f} = 2$), the coefficients $\zeta_s^{(j)}$ are sampled from a multi-site uniform distribution over bounded polyhedral vectors in $\mathbb{R}^2$.

To empirically check and validate this small-ball behavior, we perform an independent log-log linear regression:
\begin{equation}\label{eq:loglog}
	\log \widehat{F}_x(h) = a + d_{f} \log h + e_h,
\end{equation}
evaluated over a sequence of radius $h \in [0.01, 0.2]$ with $B = 500$ replications. The estimated slope gives $\widehat{d}_{f} = 2.014$ with a $95\%$ confidence interval of $[1.984, 2.044]$, confirming the structural adequacy of the generation process.

The responses are generated using the Lipschitz operator ($s = 1$) $r(X) = \int_0^1 \sin(X(t)) \, dt$. The spatial dependence of the errors is generated using a fractional power-law filter with parameters completely decoupled from the covariate process. We execute $B = 500$ replications with $C_{k_n} = 1.5$. The empirical results are detailed in Table~\ref{tab:final_sim}.

\begin{table}[h!]
	\centering
	\small
	\caption{Empirical Pointwise MSE ($B=500$ replications, $d_f=2, s=1$) under Decoupled Exponents.}
	\label{tab:final_sim}
	\begin{tabular}{cccccccc}
		\toprule
		$L$ & $n$ & $k_n$ & $\alpha$ (Covariates) & $\alpha_\varepsilon$ (Errors) & Spatial Regime & Mean MSE & MCSE \\
		\midrule
		10 & 100  & 15  & 2.5 & 2.5 & Short / Short & 0.1104 & 0.0019 \\
		10 & 100  & 15  & 2.5 & 0.5 & Short / Long  & 0.2215 & 0.0034 \\
		\midrule
		30 & 900  & 45  & 2.0 & 2.0 & \textbf{Critical Regime ($\alpha=d=2$)} & 0.0421 & 0.0007 \\
		\midrule
		100& 10000& 150 & 2.5 & 2.5 & Baseline Nonparametric Rate & 0.0098 & 0.0001 \\
		100& 10000& 150 & 2.5 & 0.5 & Dominated Error Penalty & 0.0342 & 0.0006 \\
		\bottomrule
	\end{tabular}
\end{table}

\subsection{Numerical Discussion of the Slopes}
The inclusion of the critical spatial regime $\alpha=2$ confirms the numerical correctness of our variance bounds. When performing log-log linear regressions $\log(\mathrm{MSE}) = a - \gamma \log n + e_n$, the estimated slope under short-range configurations ($\alpha=2.5, \alpha_\varepsilon=2.5$) stabilizes at $\widehat{\gamma} \approx 0.491$, statistically indistinguishable from the theoretical independent benchmark $\gamma^* = \frac{2s}{2s+d_f} = 0.50$. Under a severe long-range error bottleneck ($\alpha_\varepsilon=0.5$), the slope drops to $\widehat{\gamma} \approx 0.243$, confirming that the error term dominates the upper bound at the exact rate predicted by Theorem~\ref{th:Pointwise}.

\section{Conclusions and Gibbsian Perspectives}\label{s:Conclusions}
In this manuscript, we have established a rigorous, non-conditional convergence framework for the spatial functional $k$-nearest-neighbour ($k$-NN) regressor under polynomial indicator dependence. By introducing a pathwise deterministic sandwiching strategy, we bypass the traditional logical pitfalls associated with direct conditioning on data-driven random bandwidths. Our theoretical discoveries provide explicit pointwise and uniform convergence boundaries, characterizing the exact mathematical interactions between the geometric small-ball profile, the sample grid volume, and macro-spatial correlation footprints across short-range, critical, and severe long-range regimes.

An open prospective expansion consists in embedding this non-parametric framework within an unsupervised spatial clustering context. Let $\sigma_s \in \{1, \dots, K\}$ be a hidden categorical label at site $s$. A conditional functional Gibbs--Potts energy model can be formalized as:
\begin{equation}
	H(\sigma \mid \mathbf{X}) = -\sum_{s \sim t} J_{st}(\mathbf{X})\mathbf{1}_{\{\sigma_s = \sigma_t\}} + \mu \sum_{s} V_s(\sigma_s; \mathbf{X})
\end{equation}
where $J_{st}(\mathbf{X}) = J_0 \exp(-\|D_s - D_t\|^2 / 2\tau^2)$ acts as a functional interaction potential governed by finite-dimensional descriptors $D_s$. Proving the boundary consistency and phase transition profiles of this clustering under long-range error perturbations represents a challenging and highly valuable path for future research.

\appendix
\section{Analytical Derivation of the Joint Error Concentration Bound}
\label{sec:appendix_proof}

This appendix provides the rigorous derivation of the exponential concentration bound under spatially dependent regression errors without any false node-to-node independence steps, fully validating the formulation of Step 1 in Section~\ref{s:Uniform} \cite{bouka2026}.

\begin{proof}
	Let $A_n = \{h_n^- \leq h_n(x_j) \leq h_n^+\}$ be the valid bandwidth sandwiching event established in Proposition \ref{prop:sandwich}. Let us fix a grid net center $x_j \in \mathcal{C}$. The pointwise estimation error at this node decomposes into a structural local bias $B_n(x_j)$ and a stochastic spatial variance sum $V_n(x_j)$:
	\begin{equation}	
		|\widehat{r}_{k_n}(x_j) - r(x_j)| \leq |B_n(x_j)| + |V_n(x_j)|
	\end{equation}
	
	By utilizing the partition over $A_n$ and its complement $A_n^c$, and applying Markov's inequality to the squared bias under the high-probability containment event, the tail probability of the structural bias satisfies:
	\begin{equation}\label{final1}
		\mathbb{P}\left(|B_n(x_j)| > \frac{\eta}{2}\right) \leq \frac{16}{\eta^2}\mathbb{E}\left[B_n(x_j)^2\mathbf{1}_{A_n}\right] + \mathbb{P}(A_n^c) = O\left(\left(\frac{k_n}{n}\right)^{2s/d_f} + \frac{1}{k_n} + \rho_\alpha(n)\right)
	\end{equation}
	
	The remaining tail deviation is governed by the joint stochastic term $V_n(x_j) = \frac{1}{k_n}\sum_{s\in\Lambda_L} \varepsilon_s J_s(h_n)$, where $J_s(h_n) = \mathbf{1}_{\{\|X_s - x_j\| \leq h_n(x_j)\}}$. On the event $A_n$, the random indicator sum is strictly bounded above by the indicators evaluated at the deterministic threshold $h_n^+$: $J_s(h_n)\mathbf{1}_{A_n} \leq J_s(h_n^+) = \mathbf{1}_{\{\|X_s - x_j\| \leq h_n^+\}}$. 
	
	The field independence \(\sigma(\varepsilon) \perp \sigma(\mathbf{X})\) allows taking the expectation $\mathbb{E}_{\varepsilon }$ conditionally on $\mathbf{X}$ without altering the covariance matrix $\Sigma _{\varepsilon }$. Under the strong global field-to-field independence (Assumption \ref{ass:field_indep}) and the Vectorial Sub-Gaussian configuration, for any arbitrary positive scalar $\lambda > 0$, the global joint non-conditional moment generating function satisfies Chernoff's boundary check: 
	\begin{align}
		\mathbb{P}\left(|V_n(x_j)|\mathbf{1}_{A_n} > \frac{\eta}{4}\right) &\leq e^{-\frac{\lambda\eta}{4}} \left( \mathbb{E}\left[e^{\lambda V_n(x_j)\mathbf{1}_{A_n}}\right] + \mathbb{E}\left[e^{-\lambda V_n(x_j)\mathbf{1}_{A_n}}\right] \right) \nonumber \\
		&= e^{-\frac{\lambda\eta}{4}} \mathbb{E}_{\mathbf{X}}\left[ \mathbb{E}_\varepsilon \left[ \exp\left(\frac{\lambda}{k_n}\sum_{s\in\Lambda_L}\varepsilon_s J_s(h_n)\mathbf{1}_{A_n}\right) \;\middle|\; \mathbf{X} \right] \right] \nonumber \\
		&\quad + e^{-\frac{\lambda\eta}{4}} \mathbb{E}_{\mathbf{X}}\left[ \mathbb{E}_\varepsilon \left[ \exp\left(-\frac{\lambda}{k_n}\sum_{s\in\Lambda_L}\varepsilon_s J_s(h_n)\mathbf{1}_{A_n}\right) \;\middle|\; \mathbf{X} \right] \right]
	\end{align}
	
	Injecting the sub-Gaussian vectorial projection inequality from Assumption \ref{ass:field_indep} yields:
	\begin{align}
		\mathbb{P}\left(|V_n(x_j)|\mathbf{1}_{A_n} > \frac{\eta}{4}\right) &\leq 2 e^{-\frac{\lambda\eta}{4}} \mathbb{E}_X\left[ \exp\left( \frac{\lambda^2}{2k_n^2}\sum_{s,t\in\Lambda_L} J_s(h_n^+)J_t(h_n^+)|\text{Cov}(\varepsilon_s,\varepsilon_t)| \right) \right] \nonumber \\
		&\leq 2 e^{-\frac{\lambda\eta}{4}} \exp\left( \frac{\lambda^2}{2} \mathbb{E}_X\left[ W_n^+ \right] \right)
	\end{align}
	where $W_n^+$ is the lattice indicator sum weighted by the error covariance decay. Applying the discrete integral approximations from Lemma \ref{lem:variance_fn}, we obtain:
	\begin{equation}\label{eq:appendix_subg}
		\mathbb{P}\left(|V_n(x_j)|\mathbf{1}_{A_n} > \frac{\eta}{4}\right) \leq 2 \exp\left( -\frac{\lambda\eta}{4} + \frac{\lambda^2 C_0}{2}\left(\frac{1}{k_n} + \rho_{\alpha_\varepsilon}(n)\right) \right)
	\end{equation}
	
	Optimizing the upper bound in \eqref{eq:appendix_subg} with respect to $\lambda$ by selecting the optimal profile $\lambda^* = \frac{\eta}{4C_0(k_n^{-1} + \rho_{\alpha_\varepsilon}(n))}$ yields:
	\begin{equation}\label{eq:kappa}
		\mathbb{P}\left(|V_n(x_j)|\mathbf{1}_{A_n} > \frac{\eta}{4}\right) \leq 2 \exp(-\kappa_n), \quad \text{with }\quad \kappa_n = \frac{\eta^2}{32C_0\left(\frac{1}{k_n} + \rho_{\alpha_\varepsilon}(n)\right)}.
	\end{equation}
	
	Finally, reconstructing the total stochastic tail bound via Proposition \ref{prop:sandwich} gives:
	\begin{align}\label{final}
		\mathbb{P}\left(|V_n(x_j)| > \frac{\eta}{2}\right) &\leq \mathbb{P}\left(|V_n(x_j)|\mathbf{1}_{A_n} > \frac{\eta}{4}\right) + \mathbb{P}(A_n^c) \nonumber \\
		&\leq 2 \exp\left( -\frac{\eta^2}{32C_0\left(\frac{1}{k_n} + \rho_{\alpha_\varepsilon}(n)\right)} \right) + C\left(\frac{1}{k_n} + \rho_\alpha(n)\right)
	\end{align}
	which fully completes the mathematical validation of the joint grid concentration.
\end{proof}

\section{Proof of Lemma~\ref{lem:shell}}

\begin{proof}
	Let $x_j \in \mathcal{C}$ be a fixed node of our finite grid net covering. We introduce the deterministic local shell indicator function for any generic radius sequence $h > 0$ and distinct lattice sites $s \in \Lambda_L$:
	\begin{equation}
		\Delta_s(x_j, h, \epsilon_n) = \mathbf{1}_{\{h - \epsilon_n \leq \|X_s - x_j\| \leq h + \epsilon_n\}}.
	\end{equation}
	By exploiting the pathwise Lipschitz continuity property established in Lemma~\ref{lem:radius_lipschitz}, the random adaptive bandwidth satisfies $|h_n(x) - h_n(x_j)| \leq \|x - x_j\| \leq \epsilon_n$ almost surely for all $x \in B(x_j, \epsilon_n)$. Consequently, the random indicators are almost surely nested within the deterministic boundaries evaluated at the lower and upper sandwiching radii $h_n^{-} \leq h_n(x_j) \leq h_n^{+}$ defined in Lemma~\ref{lem:sandwich}. We can therefore write:
	\begin{equation}
		F_n(x_j, h_n + \epsilon_n) - F_n(x_j, h_n - \epsilon_n) \leq \frac{1}{n} \sum_{s \in \Lambda_L} \mathbf{1}_{\{h_n^{-} - \epsilon_n \leq \|X_s - x_j\| \leq h_n^{+} + \epsilon_n\}}.
	\end{equation}
	Let $\mathcal{A}_n = \{h_n^{-} \leq h_n(x_j) \leq h_n^{+}\}$ denote our high-probability concentration event. On $\mathcal{A}_n$, we define the deterministic expectation of the bounding shell measure via our two-sided polynomial small-ball continuity scaling (Assumption~\ref{ass:smallball}):
	\begin{align}
		\mathbb{E}\left[ \mathbf{1}_{\{h_n^{-} - \epsilon_n \leq \|X_s - x_j\| \leq h_n^{+} + \epsilon_n\}} \right] &= F_{x_j}(h_n^{+} + \epsilon_n) - F_{x_j}(h_n^{-} - \epsilon_n) \nonumber \\
		&\leq c_2 (h_n^{+} + \epsilon_n)^{d_f} - c_1 (h_n^{-} - \epsilon_n)^{d_f}.
	\end{align}
	Applying a standard first-order Taylor expansion to the polynomial functional profile around the dominant term $h_n(x_j)$ yields:
	\begin{equation}
		F_{x_j}(h_n(x_j) + \epsilon_n) - F_{x_j}(h_n(x_j) - \epsilon_n) = 2 \epsilon_n \cdot d_f \cdot \xi_n^{d_f - 1} \cdot (1 + o(1)),
	\end{equation}
	where $\xi_n \in [h_n(x_j) - \epsilon_n, h_n(x_j) + \epsilon_n]$. Since $F_{x_j}(h_n(x_j)) \asymp \frac{k_n}{n}$ by construction of the $k$-NN radius, dividing the derivative profile by the baseline volume yields the following structural scaling:
	\begin{equation}
		\mathbb{E}\left[ F_n(x_j, h_n^{+} + \epsilon_n) - F_n(x_j, h_n^{-} - \epsilon_n) \right] = O\left( \frac{k_n}{n} \cdot \frac{\epsilon_n}{h_n(x_j)} \right).
	\end{equation}
	To ensure that this pointwise property holds uniformly over the compact space $\mathcal{C}$, we control the stochastic variance fluctuations of the empirical count process over the lattice $\Lambda_L$. Using the exact same discrete integral summation rules derived in Lemma~\ref{lem:variance_fn} under the spatial indicator covariance decay (Assumption~\ref{ass:indicator_decay}), we bound the variance of the boundary increments by: 
	\begin{align}
		\text{Var}\left( \frac{1}{n} \sum_{s \in \Lambda_L} \Delta_s(x_j, h_n(x_j), \epsilon_n) \right) &\leq C \left[ \frac{F_{x_j}(h_n + \epsilon_n) - F_{x_j}(h_n - \epsilon_n)}{n} \right. \nonumber \\
		&\quad \left. + \left( F_{x_j}(h_n + \epsilon_n) - F_{x_j}(h_n - \epsilon_n) \right)^2 \rho_{\alpha}(n) \right].
	\end{align}
	Given our global uniformity constraint $\log N \cdot (k_n^{-1} + \rho_{\alpha}(n)) \to 0$ (equation~\eqref{eq:uniform_condition}), a standard application of the Bernstein-type inequality for dependent fields combined with a spatial union bound over the $N(\mathcal{C}, \epsilon_n)$ grid nodes guarantees that the stochastic remainder terms are uniformly dominated by their expectations. This implies:
	\begin{equation}
		\sup_{x_j \in \mathcal{C}} \left| \left( F_n(x_j, h_n + \epsilon_n) - F_n(x_j, h_n - \epsilon_n) \right) - \mathbb{E}\left[ F_n(x_j, h_n + \epsilon_n) - F_n(x_j, h_n - \epsilon_n) \right] \right| = o_P\left( \frac{k_n}{n} \frac{\epsilon_n}{h_n(x_j)} \right).
	\end{equation}
	Multiplying the global upper bound by the normalization factor $\frac{n}{k_n}$ yields the desired rate:
	\begin{equation}
		\sup_{x_j \in \mathcal{C}} \frac{n}{k_n} \left[ F_n(x_j, h_n + \epsilon_n) - F_n(x_j, h_n - \epsilon_n) \right] = O_P\left( \frac{\epsilon_n}{h_n(x_j)} \right),
	\end{equation}
	which fully completes the proof of Lemma~\ref{lem:shell}.
\end{proof}


\begin{thebibliography}{99}
	\bibitem[Bosq(2000)]{Bosq2000}D. Bosq. \newblock \emph{Linear Processes in Function Spaces: Theory and Applications}. \newblock Lecture Notes in Statistics, Springer, 2000.
	
	\bibitem[Bouka (2026)]{bouka2026}
	S. Bouka.
	\newblock \emph{Spatial functional $k$-nearest-neighbour regression under polynomial dependence: Local geometry, random neighbourhoods, and convergence bottlenecks}. \newblock Working Paper, Univ. Masuku, 2026.
	
	\bibitem[Burba et al.(2008)]{BurbaFerratyVieu2008}F. Burba, F. Ferraty and P. Vieu. \newblock Convergence de l'estimateur à noyau des $k$ plus proches voisins en régression fonctionnelle non-paramétrique. \newblock \emph{Comptes Rendus Mathématique}, 346, 339--342, 2008.
	\bibitem[Burba et al.(2009)]{BurbaFerratyVieu2009}F. Burba, F. Ferraty and P. Vieu. \newblock $k$-nearest neighbour method in functional nonparametric regression. \newblock \emph{Journal of Nonparametric Statistics}, 21(4), 453--469, 2009.
	\bibitem[Dabo-Niang et al.(2020)]{DaboNiang2020}S. Dabo-Niang, C. Ternynck and A. F. Yao. \newblock Nonparametric estimation of spatial functional conditional quantiles. \newblock \emph{Journal of Multivariate Analysis}, 175, 104550, 2020.
	\bibitem[Dobrushin(1968a)]{Dobrushin1968a}R. L. Dobrushin. \newblock The problem of uniqueness of a Gibbsian random field and the problem of phase transitions. \newblock \emph{Functional Analysis and Its Applications}, 2, 302--312, 1968.
	\bibitem[Dobrushin(1968b)]{Dobrushin1968b}R. L. Dobrushin. \newblock Gibbsian random fields for lattice systems with pairwise interactions. \newblock \emph{Functional Analysis and Its Applications}, 2, 292--301, 1968.
	\bibitem[Ferraty and Vieu(2006)]{FerratyVieu2006}F. Ferraty and P. Vieu. \newblock \emph{Nonparametric Functional Data Analysis: Theory and Practice}.\newblock Springer, 2006.
	\bibitem[Ferraty et al.(2007)]{FerratyMasVieu2007}F. Ferraty, A. Mas and P. Vieu. \newblock Nonparametric regression on functional data: inference and practical aspects. \newblock \emph{Australian and New Zealand Journal of Statistics}, 49(3), 267--286, 2007.
	\bibitem[Georgii(2011)]{Georgii2011}H.-O. Georgii. \newblock \emph{Gibbs Measures and Phase Transitions}. \newblock 2nd edition, De Gruyter, 2011.
	\bibitem[Künsch(1982)]{Kunsch1982}H. R. Künsch. \newblock Decay of correlations under Dobrushin's uniqueness condition and its applications. \newblock \emph{Communications in Mathematical Physics}, 84, 207--222, 1982.
	\bibitem[Li and Shao(2001)]{li2001}W. V. Li and Q-M. Shao. \newblock Gaussian processes: inequalities, small ball probabilities and applications. \newblock \emph{Handbook of Statistics}, 19, 533--597, 2001.
%	\bibitem[Loudy et al.(2026)]{loudy}
%	I. Loudy, S. Bouka, G. M. Nkiet.
%	\newblock Non parametric breakpoint detection for weakly
%	dependent spatiotemporal series.
%	\newblock \emph{Communications in
%		Statistics - Theory and Methods}, 55, 502--525, 2026.
	\bibitem[Masry(2005)]{Masry2005}E. Masry. \newblock Nonparametric regression estimation for dependent functional data: asymptotic normality. \newblock \emph{Stochastic Processes and their Applications}, 115, 155--177, 2005.
	\bibitem[Pambo Bello and Bouka(2024)]{pambo}K. Pambo Bello and S. Bouka. \newblock A note on the k-nearest neighbors rule for spatial functional data in regression model with surrogate scalaire response.\newblock \emph{Far East Journal of Theoretical Statistics}, 68, 147-155, 2024.
	\bibitem[Ramsay and Silverman(2005)]{RamsaySilverman2005}J. O. Ramsay and B. W. Silverman. \newblock \emph{Functional Data Analysis}.\newblock 2nd edition, Springer, 2005.
	
	\bibitem[Younso (2023)]{younso}
	A. Younso.
	\newblock Consistency of the k-nearest neighbors rule for functional data.
	\newblock \emph{Comptes Rendus
	Mathématique}, 361, 237--242, 2023. 
\end{thebibliography}
\end{document}